\documentclass[12pt]{extarticle} 

\usepackage{amsfonts}   
\usepackage{amsmath}    
\usepackage{amsthm}     
\usepackage{amssymb}    
\usepackage{latexsym}   
\usepackage{bbm}        

\usepackage{tikz}       
\usepackage{graphics}   

\usepackage[a4paper, total={7in, 9in}]{geometry} 

\usepackage[margin=1cm]{caption} 
\usepackage{float}      

\usepackage{enumerate}  
\usepackage{algorithm}  
\usepackage{algorithmic}

\usepackage{color}      
\usepackage[noadjust]{cite} 
\usepackage{authblk}    
\usepackage{setspace}   

\theoremstyle{plain} 
\newtheorem{theorem}{Theorem}[section]
\newtheorem{lemma}[theorem]{Lemma}

\newtheorem{proposition}[theorem]{Proposition}

\theoremstyle{definition} 
\newtheorem{definition}[theorem]{Definition}

\theoremstyle{remark} 
\newtheorem{remark}[theorem]{Remark}

\newcommand{\cH}{\mathcal{H}}

\newcommand{\cE}{\mathcal{E}}

\newcommand{\cN}{\mathcal{N}}

\newcommand{\cB}{\mathcal{B}}

\newcommand{\cP}{\mathcal{P}}
\newcommand{\cC}{\mathcal{C}}

\newcommand{\bP}{\mathbb{P}}

\newcommand{\bE}{\mathbb{E}}

\newcommand{\R}{\mathbb{R}}
\newcommand{\ep}{\epsilon }

\newcommand{\si}{\sigma }

\newcommand{\ones}{\mathbbm{1}}
\newcommand{\one}{\mathbf{1}}

\newcommand{\cs}{\textbf{c}}
\newcommand{\Mod}{\operatorname{Mod}}
\newcommand{\Dom}{\operatorname{Dom}}

\newcommand{\Adm}{\operatorname{Adm}}
\newcommand{\Ext}{\operatorname{Ext}}
\newcommand{\MEO}{\operatorname{MEO}}

\newcommand{\co}{\operatorname{co}}
\newcommand{\cl}{\operatorname{cl}}
\newcommand{\bi}{\begin{itemize}}
\newcommand{\ei}{\end{itemize}}
\newcommand{\cBhat}{\widehat{\cB}}
\newcommand{\cBtil}{\widetilde{\cB}}

\newcommand{\lbr}{\left\{ }
\newcommand{\rbr}{\right\} }

\newcommand{\coni}{\operatorname{conic}}
\newcommand{\cI}{\mathcal{I}}
\usepackage{comment}
\numberwithin{equation}{section}
\usepackage{todonotes}

\date{}

\begin{document}
\title{\bf Matroid reinforcement and sparsification\thanks{This material is based upon work supported by the National Science Foundation under Grant n. 2154032.}}
\author[1]{Huy Truong}
\author[1]{Pietro Poggi-Corradini}
\affil[1]{\small Dept. of Mathematics, Kansas State University, Manhattan, KS 66506, USA.}

\maketitle 
\begin{abstract}
Homogeneous matroids are characterized by the property that strength equals fractional arboricity, and arise in the study of base modulus \cite{truong2024modulus}. For graphic matroids, Cunningham \cite{cunninghamoptimal} provided efficient algorithms for calculating graph strength, and also for determining minimum cost reinforcement to achieve a desired strength. This paper extends this latter problem by focusing on two optimal strategies for transforming a matroid into a homogeneous one, by either increasing or decreasing element weights. As an application to graphs, we give algorithms to solve this problem in the context of spanning trees.
\end{abstract}

\vspace{0.1in}

\section{Introduction}
The authors have studied the modulus of bases of matroids in \cite{truong2024modulus}, and have shown that base modulus is closely related to strength and fractional arboricity. Let us recall these notions. For a loopless matroid $M(E,\cI)$ on a ground set $E$ with a family of independent sets $\cI$ and a rank function $r$, let $\si \in \R^E_{>0}$ be the weights assigned to each element $e$ in $E$. For $X \subseteq E$, denote $\si(X) = \sum\limits_{e \in X} \si(e)$. Then, the {\it strength} of $M$ is defined as:
\begin{equation}\label{eq:strength-problem}
	S{_\si}(M) := \min \left\{ \frac{\si(X)}{r(E) - r(E-X)} : X \subseteq E, r(E) > r(E-X) \right\},
\end{equation}
and the {\it fractional arboricity} of $M$ is defined as:
\begin{equation}
	D_{\si}(M) := \max \left\{ \frac{\si(X)}{r(X)} : X \subseteq E, r(X)>0 \right\}.
\end{equation}
An unweighted matroid ($\si\equiv 1$) is said to be {\it homogeneous} if its strength equals its fractional arboricity. This class of matroids is also called {\it uniformly dense matroids} and it was studied extensively in the literature, see for instance \cite{catlin1992,pietrofairest,rucinski1986strongly}.
For an unweighted matroid $M$, let $\cB$ denote the family of all bases of $M$. The modulus of the base family $\cB$ is related to the minimum expected overlap ($\MEO$) problem of $\cB$ (see equation (\ref{eq:meo-unweighted})), the strength, the fractional arboricity, and the theory of principal partitions of matroids, see \cite{truong2024modulus}. 
The theory of principal partitions of graphs, matroids, and submodular systems has been developed over several decades. For an overview of this theory, we recommend the survey
paper \cite{fujishige2009theory}.  Furthermore, this theory has been generalized to weighted matroids (also see \cite{fujishige2009theory}).

In this paper, a weighted matroid $M$ with element weights $\si$ is said to be ($\si$)-{\it homogeneous} if
\begin{equation}\label{eq:homogeneous-matroid}
	S_{\si}(M) = D_{\si}(M).
\end{equation}
 In Section \ref{sec:base-weights}, we introduce the minimum expected weighted overlap problem for the base family of $M$ and generalize the results in \cite{truong2024modulus} for the case of weighted matroids. In doing this, we provide a definition for homogeneous matroids in the context of base modulus (see Definition \ref{def:homogeneous}), which is different from the one given in (\ref{eq:homogeneous-matroid}), but turns out to be equivalent. Moreover, we demonstrate that $M$ is homogeneous if and only if  $\si \in \coni (\cB)$, where $\coni (\cB)$ is the conical hull of $\cB$ (see Theorem \ref{thm:hom-coni}).

In the case of graphic matroids, let's consider an undirected, connected, and weighted graph $G= (V,E)$ with edge weights $\si \in \R^E_{>0}$. The \textit{graphic matroid} associated with the graph $G$ is the matroid $M=(E,\cI)$ where a subset $A \subseteq E$ is independent if and only if $A$ does not contain any cycles in $G$. 
For every subset $A \subset E$, its rank $f(A)$ is given by $|V(A)| - r_A$, where $V(A)$ is the set of vertices of the edge-induced subgraph $H_A$ induced by $A$, and $r_A$ is the number of connected components in $H_A$ \cite{edmondsgeedy}. Then,  the family of all spanning trees of the graph $G$ is the base family of the graphic matroid $M$ \cite{edmondsgeedy}. In \cite{cunninghamoptimal}, Cunningham  provides an efficient algorithm for computing the graph strength of $G$. Additionally, Cunningham studies the problem of strength reinforcement. Consider a scenario where increasing the weight $\si(e)$ of each edge $e \in E$ incurs a cost of $m(e)$ per unit increase. Given a threshold $s_0>0$, the {\it strength reinforcement} problem aims to find the most cost-effective method to enhance edge strengths, ensuring that the resulting graph's strength is at least $s_0$. This problem can be formulated as follows:

\begin{equation} \label{rein1}
	\begin{array}{ll}
		\underset{z \in  \R^{E}_{\geq 0}}{\text{minimize}}   & m\cdot z \\
		\text{subject to} &S_{\si+z}(G) \geq s_0 .
	 	 
	\end{array}
\end{equation}
 Cunningham \cite{cunninghamoptimal} proposes an efficient solution to this problem, which is based on a greedy algorithm for polymatroids. This requires solving $2|E|$ minimum cut problems.

This led us to study the following related problem.
Consider a weighted matroid $M=(E,\cI)$ with element weights $\si$ and assume that $M$ is not homogeneous. We are interested in identifying a minimum-cost strategy to increase edge weights  in order to transform the matroid $M$ into a homogeneous one.  Given a per-unit increasing cost $m(e) \geq 0 $ for each $e \in E$,  we introduce the {\it matroid reinforcement} problem 
\begin{equation}\label{eq:hmt1}
	\begin{array}{ll}
		\underset{ z \in \R^{E}_{\geq 0}}{\text{minimize}}   & m \cdot z \\
		\text{subject to} &\si+z \in \coni(\cB).
	 
	\end{array}
\end{equation}
Likewise, we study a minimum-cost strategy to decrease edge weights that gives rise to the following {\it matroid sparsification} problem
\begin{equation} \label{eq:hmt2}
	\begin{array}{ll}
		\underset{ z \in \R^{E}_{\geq 0}}{\text{minimize}}   & m \cdot z \\
		\text{subject to} &\si-z \in \coni(\cB).
		 
	\end{array}
\end{equation}

For a matroid $M(E,\cI)$ with the rank function $r$,  it is well-known that the rank function $r$ is a polymatroid function defined on subsets of $E$ (see Definition \ref{def:poly-func}). For any polymatroid function $f$, we associate to it a polyhedron $P_f$: 

\begin{equation}\label{eq:polymatroid}
	P_f:= \lbr x \in \R^E :x \geq 0 ,x(A) \leq f(A), \forall A \subseteq E \rbr.
\end{equation}

It has been shown that $P_f$ is a polymatroid, see Definition \ref{def:polymatroid}. We also define the base polytope $B_f$ of the polymatroid function $f$ as follows:
\begin{equation}\label{eq:basepolytope}
	B_f:= \lbr x \in \R^E :x \in P_f, x(E) = f(E) \rbr. 
\end{equation}

 If $f$ is the rank function of a matroid $M(E,\cI)$ with the base family $\cB$, Edmonds \cite[Theorem 39]{edmondsgeedy} showed that the set of vertices of the polymatroid $P_f$ associated with $f$ precisely corresponds to the set of incidence vectors of $\cI$. Edmonds \cite[Theorem 43]{edmondsgeedy} also demonstrated that the set of vertices of the base polytope $B_f$ is  the set of incidence vectors of $\cB$. In other words, the base polytope $B_f$ is the convex hull of incidence vectors of $\cB$. For any positive real number $h$, we define $	hB_f:= \lbr hx :x \in B_f \rbr. $ Then, we have \[\displaystyle \coni(\cB) = \bigcup_{h>0} hB_f.\]

Here are the main results of this paper.
Let $f$ be an arbitrary polymatroid function. In Section \ref{sec:rein-poly}, we introduce a generalized version of problem (\ref{eq:hmt1}), the so-called  {\it polymatroid reinforcement} problem: 

\begin{equation}\label{eq:rein-poly}
	\begin{array}{ll}
		\underset{ z \in \R^{E}_{\geq 0}}{\text{minimize}}   & m \cdot z \\
		\text{subject to} &\si+z \in hB_f  \text{ for some } h>0.
	\end{array}
\end{equation}
In Theorem \ref{thm:rein-main}, we show that the  polymatroid reinforcement problem  (\ref{eq:rein-poly}) is equivalent to 
\begin{equation}\label{eq:rein-poly-reduced}
	\begin{array}{ll}
		\underset{ z \in \R^{E}_{\geq 0}}{\text{minimize}}   & m \cdot z \\
		\text{subject to} &\si+z \in \alpha B_f,
	\end{array}
\end{equation}
where  
\begin{equation}\label{eq:alpha}
\alpha := \min \left\{ h>0: \si \in hP_f  \right\}.
\end{equation}

In Section \ref{sec:spar-poly}, we also introduce a generalized version of problem (\ref{eq:hmt2}),  the so-called {\it polymatroid sparsification} problem:
\begin{equation}\label{eq:spar-poly}
	\begin{array}{ll}
		\underset{ z \in \R^{E}_{\geq 0}}{\text{minimize}}   & m \cdot z \\
		\text{subject to} &\si-z \in hB_f  \text{ for some } h>0.
	\end{array}
\end{equation}
We define the set function $g$ as follows: $g(U):=f (E) - f (E \setminus U )$ for all subsets $U \subseteq E$. Then, we associate the following polyhedron with $g$:
 
 \begin{equation}\label{eq:contrapoly}
 	Q_g:= \lbr x \in \R^E :x \geq 0,x(A) \geq g(A), \forall A \subseteq E \rbr.
 \end{equation} In the literature, see for instance \cite{combinatorialoptimization}, the polyhedron $Q_g$ is known as a contrapolymatroid. We also define
 \begin{equation}\label{eq:base-contrapoly}
 	C_{g}:= \lbr x \in Q_{g} : x(E) = g(E)\rbr.
 \end{equation}
 It is well-known that $B_f = C_g$ \cite[Section 44.5]{combinatorialoptimization}. Let $c$ be a constant such that $c \geq \max \lbr f(E) ,   \Vert \si \Vert_{\infty}\rbr $ and  denote $\mathbf{c}$ to be the vector of all $c$. We introduce the following polytope: 
 \begin{equation}\label{eq:contrapoly-c}
 	Q_{g,c} := \left\{ x \in \mathbb{R}^E :  \mathbf{c} \geq x, x(A) \geq g(A), \forall A \subseteq E \right\}.
 \end{equation}
Here are the main results in this direction:

In Theorem \ref{thm:translation},
we define a map  $t$: $\R^{E} \rightarrow \R^{E}$ by $ t(x) = -x + \mathbf{c}$, $x \in \R^{E}$, and show that $t(Q_{g,c})$ is a polymatroid.
Using this, in Theorem \ref{thm:spar-main}, we show that  the  polymatroid  sparsification problem is equivalent to 
 \begin{equation}\label{eq:spar-poly-reduced}
 	\begin{array}{ll}
 		\underset{ z \in \R^{E}_{\geq 0}}{\text{minimize}}   & m \cdot z \\
 		\text{subject to} &\si-z \in \beta C_g,
 	\end{array}
 \end{equation}
 where 
 \begin{equation}\label{eq:beta}
 \beta := \max \left\{ h>0: \si \in hQ_{g}  \right\}.
 \end{equation}

When $f$ is the rank function of a matroid $M(E,\cI)$, using (\ref{eq:alpha}) and (\ref{eq:beta}), we get $\alpha =D_{\si}(M)$ and $\beta =S_{\si}(M)$. Therefore, we can generalize the relationship between strength and fractional arboricity (see Theorem \ref{thm:sd-matr}) in the context of polymatroids. In particular, in Theorem \ref{thm:hab}, we demonstrate that $\alpha  \geq \beta$, and $\si \in hB_f=hC_g$ if and only if $h = \alpha=\beta$.

Finally, Theorem \ref{thm:rein-main} shows that we can reduce the feasible set of the matroid reinforcement problem.
Moreover, the reduced feasible set contains every $z \in \R^E_{\geq 0}$ such that  $\si + z$ is maximal in $\alpha P_f$. Similarly, Theorem \ref{thm:spar-main} shows that we can reduce the feasible set of 
the matroid sparsification problem and the reduced feasible set contains
 $z \in \R^E_{\geq 0}$ such that $\sigma - z$ is minimal 
in $\beta Q_g$. Consequently, we may use the greedy algorithm to solve both problems. In Section \ref{sec:sec-mat-rein-spar}, we provide Algorithm \ref{al:hmt-in} and Algorithm \ref{al:hmt-de} for solving these two problems.

In Section \ref{sec:app-graph}, we apply Algorithm \ref{al:hmt-in} and Algorithm \ref{al:hmt-de} to the case of graphic matroids. In particular, we provide  detailed methods to implement Algorithm \ref{al:hmt-in} and \ref{al:hmt-de} using Cunningham's minimum-cut formulations. Furthermore, we provide Algorithm \ref{al0} to compute the fractional arboricity and, as a consequence, we provide Algorithm \ref{al5} for computing the spanning tree modulus using the fractional arboricity.

\section{Matroid base modulus}
\subsection{Preliminaries}
\subsubsection*{Matroid}
First, let us recall the definition of matroids. For a set $X$, we denote its
cardinality by $|X|$. If $Y$ is another set, then $X - Y$ represents
the relative complement of $Y$ in $X$.

\begin{definition}\label{def:matroid}
	Let $E$ be a finite set, let $\cI$ be a set of subsets of $E$, the set system $M(E,\cI)$ is a matroid if the  following axioms are satisfied:
	\bi
	\item[(I1)] $\emptyset \in \cI$.
	\item[(I2)] If $X \in \cI$ and $Y \subseteq X$ then $Y \in \cI$ ({\it Hereditary property}).
	\item[(I3)] If $X,Y \in \cI$ and $|X| > |Y|$, then there exists $x \in X - Y$ such that $Y \cup \left\{x	\right\} \in \cI$ ({\it Exchange property}).
	\ei
	Every set in $\cI$ is called an {\it independent set}.
\end{definition}
Let  $M(E,\cI)$ be a matroid on the ground set $E$ with the set of  independent sets $\cI$. 
The maximal independent sets are called {\it bases},  the minimal dependent sets are called {\it circuits}.
The {\it rank} function, $r : 2^E \rightarrow \mathbb{Z}_{+}$, defined on all subsets $X\subset E$ is given by:
\[r(X) := \max \left\{ |Y| : Y \subseteq X, Y \in \cI \right\}.\]
The {\it closure operator} $\cl: 2^E \rightarrow 2^E$ is a set function, defined as:
\begin{equation*}\label{eq:closure-operator}
	\cl(X) := \left\{ y \in E : r(X \cup \{ y \}) = r(X) \right\}.
\end{equation*} A set $X \subseteq E$ is said to be {\it closed} if $\cl(X) =X$.
For a subset $X \subseteq E$, let $\cC(M)$ be the family of circuits of $M$. Then, the set
\[ \cC(M \setminus X) := \{ C \subseteq E-X: C\in \cC(M) \},\]
defines the family of circuits for a matroid on $E-X$. The matroid $M \setminus X$ is called the {\it deletion} of $X$ from $M$. The {\it restriction} to $X$ in $M$ is denoted by $M|X$, and is defined as the matroid on $X$ given by $M|X := M \setminus (E - X)$.
\subsubsection*{Base modulus}
Let  $M(E,\cI)$ be a matroid on the ground set $E$ with weights $\si \in \R^E_{>0}$ assigned to each element $e$ in $E$. Let $\cB$ be the family of bases of $M$.
Each base $B \in \cB$ is associated to a {\it usage vector} $\cN(B,\cdot)^T: E\rightarrow \R_{\geq 0}$. In this paper, we define $\cN(B,\cdot)^T$ to be the indicator function of $B$. In other words, $\cB$ is associated with a $|\cB| \times |E|$ {\it usage matrix} $\cN$.
A {\it density}  $\rho\in \R^E_{\geq 0}$ is a vector such that $\rho(e)$ represents the {\it cost} of using the element $e \in E$. We define the {\it total usage cost} of each base $B$ with respect to $\rho$
\begin{equation*}\label{eq:total-usage}
	\ell_{\rho}(B) := \sum\limits_{e \in E} \cN(B,e)\rho(e)= (\cN\rho)(B).
\end{equation*}
A density $\rho \in \R^E_{\geq 0 }$ is  called {\it admissible} for $\cB$, if for all $B\in \cB$,
$ \ell_{\rho}(B) \geq 1 .$
The {\it admissible set} $\Adm(\cB)$ of $\cB$ is defined as the set of all admissible densities for $\cB$,
\begin{equation*}\label{eq:adm-set}
	\Adm(\cB) := \left\{ \rho \in \R^E_{\geq 0 }: \cN\rho \geq \one \right\}. 
\end{equation*} 
Fix $1 \leq p < \infty$, the {\it energy} of the density  $\rho$ is defined as follows
\[ \cE_{p,\si}(\rho):=\sum\limits_{e \in E}\si(e)\rho(e)^p.\] 
The {\it $p$-modulus} of $\cB$ is
\[\Mod_{p,\si}(\cB):= \inf\limits_{\rho \in \Adm(\cB)} \cE_{p,\si}(\rho).\]
When $\si$ is the vector of all ones, we omit $\si$ and write $\cE_{p}(\rho) := \cE_{p,\si}(\rho)$ and $\Mod_{p}(\cB) :=\Mod_{p,\si}(\cB)$.

\subsubsection*{Fulkerson dual families and the $\MEO$ problem}
 
We will routinely identify $\cB$ with the set of its usage vectors $\left\{ \cN(B,\cdot )^T: B \in \cB \right\}$ in $\R^E_{\geq 0}$. We define the {\it dominant} of $\cB$ as
\begin{equation}
\Dom(\cB):= \co(\cB) + \R^E_{\geq 0} 
\end{equation}
where $\co(\cB)$ denotes the convex hull of $\cB$ in $\R^E$. 
Next, we recall Fulkerson duality for modulus.
The {\it Fulkerson blocker family} $\widehat{\cB}$ of $\cB$ is defined as the set of all the extreme points of $\Adm(\cB)$.
	\[ \widehat{\cB} := \Ext\left(\Adm(\cB)\right) \subset  \R^E_{\geq 0}.\]
Then, Fulkerson blocker duality \cite{fulkersonblocking} states that
\begin{equation}\label{eq:dom-adm-block-hat}
	\Dom(\widehat{\cB})= \Adm(\cB),
\end{equation}
\begin{equation}\label{eq:dom-adm-block}
	\Dom(\cB)= \Adm(\widehat{\cB}).
\end{equation}
	Let $\widetilde{\cB}$ be a set of vectors in $\R^E_{\geq 0}$. We say that $\cB$ and $\widetilde{\cB}$ are a {\it Fulkerson dual pair} (or $\cBtil$ is a {\it Fulkerson dual family} of $\cB$) if \[\Adm(\cBtil) = \Dom(\cB).\] 
 Then, $\cBhat$ is the smallest Fulkerson dual family of $\cB$, meaning that $\cBhat \subset \cBtil$ \cite{huyfulkerson}.

When  $1<p < \infty$, let $q:=p/(p-1)$ be the Hölder conjugate exponent of $p$. For any set of weights $\si \in \R^E_{>0}$,  define the dual set of weights $\widetilde{\si}$ as $\widetilde{\si}(e):=\si(e)^{-\frac{q}{p}}$ for all $e\in E$. Let $\cBtil$ be a Fulkerson dual family of $\cB$. Fulkerson duality for modulus  \cite[Theorem 3.7]{pietroblocking} states that
\begin{equation}
	\Mod_{p,\si}(\cB)^{\frac{1}{p}}\Mod_{q,\widetilde{\si}}(\widetilde{\cB})^{\frac{1}{q}}=1.
\end{equation}
Moreover, the optimal $\rho^*$ of $\Mod_{p,\si}(\cB) $ and the optimal $\eta^*$ of $\Mod_{q,\widetilde{\si}}(\widetilde{\cB})$ always exist, are unique, and are related as follows,
\begin{equation}\label{eq:weighted-eta-rho}
	\eta^{\ast}(e) = \frac{\si(e)\rho^{\ast}(e)^{p-1}}{\Mod_{p,\si}(\cB)}, \quad \forall e\in E.
\end{equation}
When $p=2$, we have
\begin{equation}\label{eq:mod2}
	\Mod_{2,\si}(\cB)\Mod_{2,\si^{-1}}(\widetilde{\cB})=1 \qquad\text{and}\qquad   \displaystyle \eta^{\ast}(e) = \frac{\si(e)}{\Mod_{2,\si}(\cB)}\rho^{\ast}(e) \quad \forall  e\in E.
\end{equation}
Let $\cP(\cB)$ be the set of all probability mass functions (pmf) on $\cB$. According to the probabilistic interpretation of modulus \cite{pietrominimal}, we can express

\begin{equation}\label{eq:modmeo}
	\Mod_2(\cB)^{-1} = \min\limits_{\mu \in \cP(\cB)} \mu^T\cN\cN^T\mu.
\end{equation}
Consider the scenario where $\cB$ is a collection of subsets of $E$ with usage vector given by the indicator function. Given a pmf $\mu \in \cP(\cB)$, let $\underline{B}$ and $\underline{B}'$ be two independent random bases in $\cB$, identically distributed with law $\mu$. 
The cardinality of the overlap between $\underline{B}$ and $\underline{B}'$, is $|\underline{B} \cap \underline{B}'|$ and is a random variable whose expectation is denoted by  $\bE_\mu|\underline{B}\cap\underline{B}'|$, which equals $\mu^T\cN\cN^T\mu$. Then, the {\it minimum expected overlap} ($\MEO$) problem for $\cB$ is formulated as 
\begin{equation}\label{eq:meo-unweighted}
	  \min\limits_{\mu \in \cP(\cB)} \bE_\mu|\underline{B}\cap\underline{B}'|.
\end{equation} 
	Moreover, any pmf $\mu \in  \cP(\cB)$ is optimal if and only if
	\begin{equation}\label{eq:eta-rho}
		(\cN\mu)(e) = \rho^*(e)/\Mod_2(\cB) \quad \forall e \in E.
	\end{equation}

\subsection{Matroid base modulus with weights}\label{sec:base-weights}
Given a matroid $ M(E,\cI) $ with weights $\si \in \R^E_{>0}$, let $ \cB $ be the base family of $ M $. A set $ X \subseteq E $ is said to be complement-closed if $E - X$ is closed. Let $ \Phi $ be the family of all nonempty complement-closed sets $ X \subseteq E $ with usage vectors:
\begin{equation}\label{usage2}
	\widetilde{\cN}(X,\cdot)^T := \frac{1}{r(E) - r(E - X)}\ones_{X}.
\end{equation}
Then, $ \Phi $  is a Fulkerson dual family of $\cB$ \cite[Theorem 6.2]{truong2024modulus}.

Base modulus for unweighted matroids was thoroughly studied in \cite{truong2024modulus}.
In this section, we generalize the results of base modulus for unweighted matroids to weighted matroids. We start by considering a matroid $M(E,\cI)$ 	with weights $\si \in \mathbb{Z}_{>0}^{E}.$ 

For each $e \in E(M)$, we create a set $X_e$ containing $\si(e)$ copies of $e$. Let $X_e := \lbr e_1,e_2,\dots, e_{\si(e)} \rbr$ be a set such that $X_e \cap X_{e'} = \emptyset,$ for all $ e, e' \in E(M)$ with $e \neq e'$. Next, we define the $\si$-{\it parallel extension} $M_{\si} $ of $M$ as follows. The $\si$-parallel extension $M_{\si} $ is obtained by replacing each element $e \in E(M)$ by $X_e$. In specific, the ground set $E(M_{\si})$ of $M_{\si}$ is $\bigcup_{e \in E(M)}X_e$. A subset $Y \in E(M_{\si})$ is independent in $M_{\si}$ if and only if $\forall e \in E(M), |X_e \cap Y| \leq 1$ and the set $\lbr e \in E(M): X_e \cap Y \neq \emptyset \rbr$ is independent in $M$. 
In the case where every element $e$ is assigned a constant weight $r$, we write $M_r$ for $M_{\si \equiv r}$ and call $M_r$ the $r$-parallel extension of $M$.

 Let $E' = \lbr e_1: e\in E(M)\rbr \subseteq E(M_{\si}).$ There is a matroid isomorphism between $M$ and $M_{\si}|E'$ with the bijection $e \leftrightarrow e_1$ between $E(M)$ and $E'$. Thus, we can see $M$ as the restriction $M_{\si}|E'$ of $M_{\si}$. Moreover, based on the definition of $M_{\si}$, we have that $S_{\si}(M) = S(M_{\si})$, $D_{\si}(M) = D(M_{\si})$, and $\Mod_{p,\si}(\cB(M)) = \Mod_{p}(\cB(M_{\si}))$. Then, by continuity and $1$-homogeneity of modulus (meaning that $\Mod_{p,k\si}(\cB(M)) = k\Mod_{p,\si}(\cB(M))$ for $k>0$), it is then straightforward to generalize our results to weighted matroids with weights $\si \in \R^E_{>0}$. In the rest of this section, we give results without proofs for weighted base modulus generalized from the unweighted case.

We begin by introducing the {\it minimum expected weighted overlap} problem. Let $\cP(\cB)$ be the set of all probability mass functions (or pmf) on $\cB$. Given a pmf $\mu \in \cP(\cB)$, let $\underline{B}$ and $\underline{B'}$ be two independent random bases, identically distributed by the law $\mu$. We measure the weighted overlap between $\underline{B}$ and $\underline{B'}$,
\begin{equation}
	\si^{-1}(\underline{B} \cap \underline{B'}) := \sum\limits_{e \in \underline{B} \cap \underline{B'}} \si^{-1}(e),
\end{equation}
which is a random variable whose expectation is denoted by  $\bE_{\mu} \left[ \si^{-1}(\underline{B} \cap \underline{B'}) \right]$. Then, the {\it minimum expected weighted overlap} ($\MEO_{\si^{-1}}$) problem is the following problem: 

\begin{equation} \label{meo}
	\begin{array}{ll}
		\text{minimize}    &\bE_{\mu} \left[ \si^{-1}(\underline{B} \cap \underline{B'}) \right] \\
		\text{subject to } & \mu \in \cP(\cB_G). 	 
	\end{array}
\end{equation}
Next, we present a theorem that characterizes the relationship between base modulus and the $\MEO_{\si^{-1}}$ problem.
\begin{theorem} Given a matroid $ M(E,\cI) $ with weights $\si \in \R^E_{>0}$. Let $\cB $ be the base family with usage vectors given by the indicator functions and let $\widetilde{\cB}$ be a Fulkerson dual family of $\cB$. Then, $\rho \in  \R^E_{\geq 0}$, $\eta \in  \R^E_{\geq 0}$ and $\mu \in \bP(\cB)$ are optimal respectively for $\Mod_{2,\si}(\cB)$, $\Mod_{2,\si^{-1}}(\widetilde{\cB})$ and $\MEO_{\si^{-1}}(\cB)$ if and only if the following conditions are satisfied.
	
	\begin{align*}
		{(i)} &\qquad \rho \in \Adm(\cB), \qquad \eta = \cN^T\mu,\\
		{(ii)} &\qquad \eta(e) = \frac{\si(e)\rho(e)}{\Mod_{2,\si}(\cB)} \qquad \forall e \in E,\\
		{(iii)} &\qquad \mu(B)(1-\ell_\rho(B)) = 0 \qquad \forall B \in \cB.
	\end{align*}
	In particular, 
	\begin{align}
		\MEO_{\si^{-1}}(\cB) = \Mod_{2,\si}(\cB)^{-1} = \Mod_{2,,\si^{-1}}(\widetilde{\cB}).
	\end{align}
\end{theorem}
Next, let us introduce the definition of a homogeneous matroid.
\begin{definition}\label{def:homogeneous}
	Given a matroid $ M(E,\cI) $ with weights $\si \in \R^E_{>0}$. Let $\cB $ be the base family with usage vectors given by the indicator functions. Let $\widetilde{\cB}$ be a Fulkerson dual family of $\cB$. Let $\rho^*$ and $\eta^*$ be the unique optimal densities for $\Mod_{2,\si}(\cB)$ and $\Mod_{2,\si^{-1}}(\widetilde{\cB})$ respectively. Then, the matroid $M$ is said to be {\it homogeneous} if $\si^{-1}\eta^*$ is constant, or equivalently, $\rho^*$ is constant.
\end{definition}

Several properties of the weighted base modulus can be proved in the same manner as those for the weighted spanning tree modulus on graphs. See \cite{huyfulkerson} for details on the weighted spanning tree modulus.

\begin{theorem}\label{whom}
	Given a matroid $ M(E,\cI) $ with weights $\si \in \R^E_{>0}$. Let $\cB $ be the base family of $M$. Let $\widetilde{\cB}$ be a Fulkerson dual family of $\cB$.  Define the density $n_{\si}$:
	
	\begin{equation}\label{nhom}
		n_{\si}(e) :=\frac{\si(e)}{\si(E)}r(E) \qquad \forall e \in E.
	\end{equation} 
	Then, $M$ is homogeneous if and only if  $\eta_{\si} \in \Adm(\widetilde{\cB})$.
\end{theorem}

\begin{theorem}\label{thm:hom-coni}
 Given a matroid $ M(E,\cI) $ with weights $\si \in \R^E_{>0}$. Let $\cB $ be the base family of $M$. Then, $M$ is homogeneous if and only if  $\si \in \coni (\cB)$, where $\coni (\cB)$ is the conical hull of $\cB$.
\end{theorem}
\begin{proof}
The proof is the same as Proof of Theorem 7.5 in \cite{huyfulkerson}.
\end{proof}
\begin{theorem} Given a matroid $ M(E,\cI) $ with weights $\si \in \R^E_{>0}$. Let $\cB $ be the base family of $M$. Let $S_{\si}(M)$ be the strength of $M$. Let $\widetilde{\cB}$ be a Fulkerson dual family of $\cB$.
	Let $\eta^*$  be the optimal density for $\Mod_{2,\si^{-1}}(\widetilde{\cB})$. Denote
	\begin{equation}\label{eq:emax}
		E_{ max} :=\left\{ e\in E: \si^{-1}(e)\eta^*(e)=\max\limits_{e \in E}\si^{-1}(e)\eta^*(e) =: (\si^{-1}\eta^*)_{  max} \right\}.
	\end{equation}
	Then, $ E_{ max}$ is optimal for the strength of $M$, and
	\begin{equation}
		(\si^{-1}\eta^*)_{  max} = \frac{1}{S_{\si}(M)}.
	\end{equation}
\end{theorem}

\begin{theorem}\label{emin}
	Given a matroid $ M(E,\cI) $ with weights $\si \in \R^E_{>0}$. Let $\cB $ be the base family of $M$.  Let $D_{\si}(G)$ be the fractional arboricity of $M$. Let $\widetilde{\cB}$ be a Fulkerson dual family of $\cB$.
	Let $\eta^*$  be the optimal density for $\Mod_{2,\si^{-1}}(\widetilde{\cB})$.  Denote 
	\begin{equation}\label{eq:emin}
		E_{ min} :=\left\{ e\in E: \si^{-1}(e)\eta^*(e)=\min\limits_{e \in E}\si^{-1}(e)\eta^*(e) =: (\si^{-1}\eta^*)_{  min} \right\}.
	\end{equation}
	Then, $ E_{ min}$ is optimal for the fractional arboricity of $M$, and
	\begin{equation}
		(\si^{-1}\eta^*)_{  min} = \frac{1}{D_{\si}(M)}.
	\end{equation}
\end{theorem}
\begin{theorem}\label{thm:sd-matr}
Given a matroid $ M(E,\cI) $ with weights $\si \in \R^E_{>0}$. Let $\cB $ be the base family of $M$.   Let $S_{\si}(M)$ be the strength of $M$. Let $D_{\si}(G)$ be the fractional arboricity of $M$.  Let $\widetilde{\cB}$ be a Fulkerson dual family of $\cB$.
	Let $\eta^*$  be the optimal density for $\Mod_{2,\si^{-1}}(\widetilde{\cB})$. Then,
	\begin{equation}\label{maxmin}
		\frac{1}{(\si^{-1}\eta^*)_{ max}} = S_{\si}(G)\leq \frac{\si (E)}{r(E)}\leq D_{\si}(G) = \frac{1}{(\si^{-1}\eta^*)_{ min}}.
	\end{equation}
	Moreover, the matroid $M$ is homogeneous if and only if $S_{\si}(M) = D_{\si}(M)$.
\end{theorem}

\section{Polymatroid reinforcement and polymatroid sparsification}\label{sec:polys}
\subsection{Preliminaries}
Polymatroids were introduced by Edmonds \cite{edmonds2003submodular}, who also established many profound insights into the topic. Let $E$ be a finite set, we recall the definition of polymatroids as follows.

\begin{definition}\label{def:polymatroid}
	Let $ P\subseteq \mathbb{R}^E$. Then $P$ is called a polymatroid if $P$ is a compact non-empty subset of  $\mathbb{R}^E_{\geq 0}$ satisfying the following properties:
	\begin{itemize}
		\item[(i)] If $0 \leq x \leq y$ and $y \in P$, then $x \in P$.
		\item[(ii)] For any $x \in \mathbb{R}^E_{\geq 0}$, any maximal vector $y \in P$ with $y \leq x$ (such $y$ is called a $P$\textit{-basic} of $x$) has the same component sum $y(E)$.
	\end{itemize}
\end{definition}

Let $f$ be a real-valued function defined on subsets of $E$. The function $f$ is said to be {\it submodular} if for all subsets $A,B$ of $E$, we have
\begin{equation}
	f(A \cap B) + f(A \cup B) \leq f(A) + f(B).
\end{equation}
Conversely, $f$ is said to be {\it supermodular} if $-f$ is submodular. Next, we recall the definition of {\it polymatroid functions}.

\begin{definition}\label{def:poly-func}
	A polymatroid function is a real-valued function $f$ defined on subsets of $E$, which is normalized, nondecreasing, and submodular, meaning that
	
	\bi
	\item[(i)] $f(\emptyset)=0$;
	\item[(ii)] $f(A) \leq f(B)$ if $A \subseteq B \subseteq E$;
	\item[(iii)] $f(A \cap B)+f(A \cup B) \leq 	f(A) +f(B)$ for all subsets $A,B \subseteq E$.
	\ei
\end{definition}
Let $P_f$ be the polyhedron associated with a polymatroid function $f$ defined as in (\ref{eq:polymatroid}).
It is well-known that $P_f$ is a polymatroid. It is important to note that every polymatroid $P$ can be represented in the form of $P_f$ for some polymatroid function $f$. We let $B_f$ be the base polytope of the polymatroid function $f$ defined as in (\ref{eq:basepolytope}).

Let $P$	 be a polymatroid,  one interesting problem in this field is finding a $P$-basis $y$ of a given vector $x \in \mathbb{R}^E_{\geq 0}$. An algorithm for solving this problem is a greedy algorithm which starts with $y = 0$ and successively increases each component of $y$ as much as possible, subject to the constraints $y \leq x$ and $y \in P$, see \cite{edmondsgeedy}.
Given a vector $m \in \mathbb{R}^E$, the greedy algorithm can be generalized to solve the following optimization problem:
\begin{equation}\label{pbasic}
	\min\left\{m \cdot y : y \text{ is a P-basis of }x \right\},
\end{equation}
as discussed in \cite{cunninghamoptimal}, \cite{edmondsgeedy}.
For (\ref{pbasic}), we increase successively each component in the order $j_1,j_2,\dots,j_k$ where $k =|E|$ and $m_{j_1} \leq m_{j_2} \leq \dots \leq m_{j_k}$. 
The well-known greedy algorithm for finding a minimum spanning tree in a graph (Kruskal's algorithm)  is a special case of this greedy algorithm. It arises by choosing $x = \one\in \R^E$ where $\one$ is the vector containing all ones. Cunningham \cite{cunninghamoptimal} proved the generalized algorithm does solve (\ref{pbasic}). To implement this algorithm, we need an oracle that computes

\begin{equation}\label{oracle1}
	\max \left\{ \epsilon : y+\epsilon\ones_{\left\{ j\right\}} \in P \right\} \qquad \text{ for any } y \in P \text{ and }j\in E.
\end{equation}

\subsection{Polymatroid reinforcement }\label{sec:rein-poly}

Let $f$ be a polymatroid function on $E$. Let  $P = P_f$ be the associated polymatroid with $f$. Let $B_f$ be the base polytope defined as in (\ref{eq:basepolytope}). For $h>0$, we define $hB_f:= \lbr hx :x \in B_f \rbr$, and	$hP_f := \left\{ hx: x \in P_f \right\}$. Then, we have that $hP_f$ is also a polymatroid with the polymatroid function $hf$, and $hB_f$ is its base polytope. Given a per-unit increasing cost $m(e) \geq 0 $ for each $e \in E$ and  $s \in \R^E_{>0}$. We recall the polymatroid reinforcement problem in (\ref{eq:rein-poly}) where $\si=s$. To study this problem, we start by characterizing maximal vectors in the associated polymatroid $P = P_f$.

\begin{lemma}\label{maximal}
	Let  $P = P_f$ be the associated polymatroid of a polymatroid function $f$. Let $B_f$ be the base polytope of $f$. Given $x \in P$. Then, $x$ is a  maximal vector in $P$ if and only if $x \in B_f$.
\end{lemma}
This fact should be standard, but we provide a proof for the reader's convenience.
\begin{proof}
	Assume that $x \in B_f$. Then, we have $x(E) = f(E)$. Hence, $x$ is maximal because if we increase any component of $x$, then  $x(E) > f(E)$, contradiction with the definition of $P$ in (\ref{eq:polymatroid}).
	
	Assume that  $x \in P$ is a  maximal vector in $P$. Let $c$ be a number such that $c \geq f(E)$. Let $ \cs \in \R^E$ be the column vector of all $c$. 
	We have 
	\begin{align*}
		x(e) \leq f(\left\{ e\right\}) \leq c, \quad \forall e \in E  &&\left(\text{by the definition (\ref{eq:polymatroid}})\right).
	\end{align*}
	Hence, we obtain $x \leq \cs $. Combine with the fact that  $x$ is  maximal in $P$, we achieve that $x$ is a $P$-basis of $\cs$. Let $y \in B_f$, then $y$ is maximal and is a  $P$-basis of $\cs$. By  Definition \ref{def:polymatroid} of polymatroid, $x$ and $y$ has the same component sum, in other words, $x(E)=y(E)= f(E).$ 	Thus, we obtain $x \in B_f$.
\end{proof}

The following lemma defines a polymatroid generated by any element $s\in P_f$.

\begin{lemma}\label{pspoly}
	Let  $P = P_f$ be the associated polymatroid of a polymatroid function $f$. Let $B_f$ be the base polytope of $f$. Given $s \in P$, we define the following polyhedron
	\begin{equation}\label{eq:restrict-poly}
		P^s := \left\{ x \in \R^E_{\geq 0 } : s + x \in P \right\}. \end{equation}
	Then, $P^s$ is a polymatroid.
	
\end{lemma}

\begin{proof}
	We aim to show that $P^s$ satisfies Definition \ref{def:polymatroid}.
	Since $s \in P$, then $P^s$ contains the vector $0$. Moreover, $P^s$ is compact by its definition.
	For any vectors $x,y$ such that $ x \leq y$ and $y \in P^s$, we have $s+y \in P$ and  $s+x \leq s+y$. Then,  $s+x \in P$ by Definition \ref{def:polymatroid} (i), which implies $x \in P^s$. Hence, $P^s$ satisfies Definition \ref{def:polymatroid} (i).

	For any $y \in \mathbb{R}^E_{\geq 0}$, let $x$ be a maximal vector among all vectors satisfying $x \leq y$ and $x \in P^s$. By the definition of $P^s$, $s+x \in P$.  We want to show that: $s+x$ is a $P$-basis of $s+y$. Indeed, suppose that $s+x$ is not a $P$-basis of $s+y$. Then, there exists $\epsilon > 0$ and $e \in E$ satisfying
	\begin{equation*}
		s+x+\epsilon \ones_{\{ e \}} < s+y \quad \text{and} \quad s+x+\epsilon \ones_{\{ e \}}
		\in P.
	\end{equation*}
	This is equivalent to
	\begin{equation*}
		x+\epsilon \ones_{\{ e \}} < y \quad \text{and} \quad x+\epsilon \ones_{\{ e \}}
		\in P^s,
	\end{equation*}
	contradicting the definition of $x$. Thus, $s+x$ is a $P$-basis of $s+y$. By Definition \ref{def:polymatroid} (ii), we have $(s+x)(E)$ does not depend on $x$, it only depends on $y$ and $s$.
	This implies that $x(E) = (s+x)(E) - s(E)$ does not depend on $x$. Therefore, $P^s$ satisfies Definition \ref{def:polymatroid} (ii). In conclusion, $P^s$ is a polymatroid.
\end{proof}

The following lemma characterizes maximal vectors in the polymatroid $P^s$.

\begin{lemma}\label{maximalps}
	Let  $P = P_f$ be the associated polymatroid of a polymatroid function $f$. Let $B_f$ be the base polytope of $f$. Given $s \in P$, let $P^s$ be defined as in (\ref{eq:restrict-poly}). Given $x \in P^s$, then, $x$ is maximal in $P^s$ if and only if $(s+x)(E)= f(E)$. 
\end{lemma}

\begin{proof} Assume that $(s+x)(E)= f(E)$. By Lemma \ref{maximal}, $s+x$ is maximal in $P$, this implies that $x$ is maximal in $P^s$ by the definition of $P^s$.
	
	Conversely, assume that $x$ is maximal in $P^s$. To prove that $(s+x)(E)= f(E)$, we first claim that there exists $x' \in P^s$ such that   $(s+x')(E)= f(E)$. 
	
	If $s$ is maximal in $P$, by Lemma \ref{maximal}, we obtain $s(E)=f(E)$. Then, we choose $x'=0 \in \R^E$. If $s$ is not maximal in $P$, then we can increase successively each component of $s$ as much as possible subject to $s \in P$, the resulting vector has the form $s+x'$ for some $ x' \geq 0$. The vector $s+x'$ is maximal in $P$ because if there exists some $e \in E$ and $\ep >0$ such that $(s+x'+\epsilon \ones_{\{ e \}}) \in P$, then during the process, we must have been increased the $e$-component of $s$ to at least $x'(e)+\ep$, a contradiction.
	Hence, by Lemma \ref{maximal}, $(s+x')(E)= f(E)$ . Therefore, there exists $x' \in P^s$ such that   $(s+x')(E)= f(E)$. 
	
	By the above argument, $x'$ is also maximal in $P^s$. By the Definition \ref{def:polymatroid} (ii), $x$ and $x'$  has the same component sum, this implies that $(s+x)(E)= (s+x')(E)= f(E)$.
\end{proof}

The feasible set of the reinforcement problem is $	\displaystyle  \bigcup_{ h>0} A_h$ where 

\begin{equation}\label{eq:def-eh}
	A_h :=  \left\{  z\in \R^E_{\geq 0}:  s+ z \in hB_f \right\}.
\end{equation}

\begin{lemma}\label{CD}
	Let  $P = P_f$ be the associated polymatroid of a polymatroid function $f$. Let $B_f$ be the base polytope of $f$. Let $A_h$ be defined as in (\ref{eq:def-eh}). Let
		$\alpha := \min \left\{ h>0: s \in hP_f  \right\}$ be defined as in (\ref{eq:alpha}).
	Then, $A_h \neq \emptyset$ if and only if $h \geq \alpha $.
\end{lemma}
\begin{proof}
	Let $h >0$. If there exists $z\in A_h$. Then, we have $ s+ z \in hB_f \subset hP_f.$ Because $s \leq s+z$, we have $s \in hP_f$, hence $h \geq \alpha $ by the definition of $\alpha $. If $h \geq \alpha $, then $0 \in A_h$.
\end{proof}

\begin{remark}\label{ehandps}
	Let $\alpha $ be defined as in (\ref{eq:alpha}). For $h \geq \alpha $, we have 
	\begin{equation*}
		A_h=  \left\{  z\in \R^E_{\geq 0}:  s+ z \in hP_f, (s+z)(E) = hf(E) \right\}.
	\end{equation*}
	By Lemma \ref{maximalps}, we obtain 
	\begin{equation}
		A_h= \left\{ z \in (hP_f)^{s} : z \text{ is  maximal in } (hP_f)^{s} \right\}.
	\end{equation}
\end{remark}
\begin{remark}
	Let $h_1 \geq h_2 \geq \alpha $, and let $z_1 \in A_{h_1}$, $z_2 \in A_{h_2}$. Then \[z_1(E) = h_1f(E) - s(E) \geq h_2f(E) - s(E) = z_2(E).\] Therefore, if the cost $m(e)  =1$ for all $e \in E$, then \[\min\limits_{z \in A} \one \cdot z = \min\limits_{z \in A_{\alpha }}  \one \cdot z = \alpha f(E) - s(E).\]
\end{remark}
Now, we consider any costs $m \in \R^{E}_{\geq 0}$. The following theorem is one of our main results.

\begin{theorem}\label{thm:rein-main} 
	Let  $P = P_f$ be the associated polymatroid of a polymatroid function $f$. Let $B_f$ be the base polytope of $f$.
	Given $s \in \R^E_{>0}$ and $m \in \R^E_{\geq 0}$. Let $A_h$, and  $\alpha $ be defined as in (\ref{eq:def-eh}), and (\ref{eq:alpha}) respectively. For every $h\geq \alpha $, let $z^*_h$ be an optimal solution of the problem $\min\limits_{z\in A_h} m \cdot z$. Then,
	\begin{equation}\label{eq:rein-main}
		m \cdot z^*_{h} \geq m \cdot z^*_{\alpha }  \qquad \forall h \geq \alpha .
	\end{equation}   
	In other words, the polymatroid reinforcement problem is equivalent to $\min\limits_{z \in A_{\alpha }} m \cdot z.$

\end{theorem}

\begin{proof}

Let $h \geq \alpha$. Assume that there exists $y \in A_{\alpha}$ such that $y \leq z^*_{h}$.
Then
\begin{equation}\label{eq:z1zh}
	m \cdot y \leq m \cdot z^*_{h},
\end{equation}
and by the definition of $z^*_{\alpha}$, we obtain
\begin{equation}\label{eq:zaz1}
	m \cdot z^*_{\alpha} \leq m \cdot y.
\end{equation}
Hence, we have (\ref{eq:rein-main}).

The rest of the proof is to show that such $y$ exists. Let $(\alpha P_f)^{s}$ be defined as in (\ref{eq:restrict-poly}), and let $y$ be an $(\alpha P_f)^{s}$-basic of $z^*_{h}$. Then, we obtain that $y \leq z^*_{h}$ and $s + y$ is an $(\alpha P_f)$-basic of $s + z^*_{h}$. The proof is completed if we can show that $y \in A_{\alpha}$, i.e., $(s + y)(E) = \alpha f(E)$. To that end, we denote
\[
x := \frac{\alpha}{h}(s + z^*_{h}) \leq (s + z^*_{h}).
\]
By the definition of $z^*_{h}$, we have $s + z^*_{h} \in hB_f$. Thus, $x \in \alpha B_f$, this implies that $x(E) = \alpha f(E)$. By Lemma \ref{maximal}, $x$ is maximal in $\alpha P_f$. Moreover, since $x \leq (s + z^*_{h})$, we have that $x$ is an $(\alpha P_f)$-basic of $s + z^*_{h}$. Note that $s + y$ is also an $(\alpha P_f)$-basic of $s + z^*_{h}$. By Definition \ref{def:polymatroid} (ii), $(s + y)(E) = x(E) = \alpha f(E)$. The proof is completed.

\end{proof}

\subsection{Polymatroid sparsification}\label{sec:spar-poly}

Given a polymatroid function $f$ on $E$, then $f$ is normalized, nondecreasing, and submodular. We define the set function $g$ as follows:

\begin{equation}\label{eq:superg}
	g(U):=f (E) - f (E \setminus U ) 
\end{equation} for all subsets $U \subseteq E$. It follows that $g(\emptyset) = f(E)-f(E) = 0$. For any set $A \subseteq B \subseteq E$, we have 

\begin{equation*}
	g(A) =f (E) - f (E \setminus A ) \leq f (E) - f (E \setminus B ) = g(B).
\end{equation*}
Furthermore, for all subsets $A,B \subseteq E$, by definition of $g$, we have

\[ g(A \cap B)+g(A \cup B) \geq g(A) +g(B).\]
Hence, the function $g$ is normalized, nondecreasing and supermodular.
Next, we associate the contrapolymatroid $Q_g$ with $g$. Notice that for any $e \in E$, we have $g(\lbr e \rbr) \geq g(\emptyset) =0$. Thus, the polyhedron $Q_g$ can be described as \[Q_g = \lbr x \in \R^E :x(A) \geq g(A), \forall A \subseteq E \rbr.\]
Let $C_g$ be defined as in (\ref{eq:base-contrapoly}).
Let $c$ be a constant such that $c \geq f(E)$. Denote $\mathbf{c}$ to be the vector of all $c$. Consider the following polyhedron:
\begin{equation}\label{eq:contrapoly-c}
	Q_{g,c} := \left\{ x \in \mathbb{R}^E :  \mathbf{c} \geq x, x(A) \geq g(A), \forall A \subseteq E \right\}.
\end{equation}
This polyhedron is bounded and hence is a polytope. Additionally, we define
\begin{equation}\label{eq:base-contrapoly-c}
	C_{g,c}:= \lbr x \in Q_{g,c} : x(E) = g(E)\rbr.
\end{equation} 
We aim to demonstrate that, under some translation, the polytope $Q_{g,c}$ maps to a polymatroid.

\begin{theorem}\label{thm:translation}
	Let  $t$: $\R^{E} \rightarrow \R^{E}$ such that $ t(x) = -x + \mathbf{c}$, $x \in \R^{E}$.  Then, $t(Q_{g,c})$ is a polymatroid.
\end{theorem}
\begin{proof}
	Let $x \in Q_{g,c}$, and $y = -x + \mathbf{c}$. Then, we have $y \geq 0$. Next, note that
	\begin{align*}
		x(A) \geq g(A) &\Leftrightarrow -y(A) + c|A| \geq f(E) - f(E \setminus A) \\
		&\Leftrightarrow y(A) \leq -f(E) + f(E \setminus A) + c|A|.
	\end{align*}
	
	We define a set function $h(A) := -f(E) + f(E \setminus A) + c|A|$. Then, $h$ is normalized, nondecreasing, and submodular. Indeed, by defintion of $h$, we have that $h$ is normalized, and submodular.
	For $A \subsetneq B \subseteq E$, we have 
	\begin{align*}
		-f(E) + f(E \setminus A) + c|A| &\leq -f(E) + f(E \setminus B) + c|B|\\
		\Leftrightarrow f(E \setminus A) - f(E \setminus B) &\leq c|B \setminus A|,
	\end{align*} which is true because $c \geq f(E) \geq f(E \setminus A)$.
	So, the function $h$ is nondecreasing. Therefore, $t(Q_{r,c})$ is a polymatroids.
\end{proof}

 In the rest of this section, we apply Theorem \ref{thm:translation} 
to provide corresponding results from Section \ref{sec:rein-poly}. 
\begin{lemma}
Let $Q := Q_{g,c}$ be defined as in (\ref{eq:contrapoly-c}). 
 A vector $x \in Q$ is  minimal in $Q$ if and only if $x(E) = g(E) = f(E)$.
\end{lemma}

\begin{lemma}
 Given $s \in Q $, then, $	Q^s := \left\{ z \in \R^E_{\geq 0 } : s - z \in Q \right\}$ is a polymatroid under some translation.
\end{lemma}
\begin{lemma}
	 Given $s \in Q $ and $z \in Q^s$ where $Q^s := \left\{ z \in \R^E_{\geq 0 } : s - z \in Q \right\}$, then, $z$ is minimal in $Q^s$ if and only if $(s-z)(E)= f(E)$.
\end{lemma}

Given a per-unit increasing cost $m(e) \geq 0 $ for each $e \in E$ and  $s \in \R^E_{>0}$. Let $c$ be a constant such that $c \geq \max \lbr f(E),\Vert s \Vert_{\infty} \rbr$. We recall the polymatroid sparsification problem in (\ref{eq:spar-poly}) where we replace $C_g$ by $C_{g,c}$ and $\si=s$. In fact, we have $C_g= C_{g,c}$ (see Theorem \ref{thm:contrapoly-prop}).  The feasible set of the  polymatroid sparsification problem is 	$\displaystyle \bigcup_{ h>0} F_h$ where $	F_h :=  \left\{  z\in \R^E_{\geq 0}:  s- z \in hC_{g,c} \right\}$.
\begin{remark}
Let
$\beta := \max \left\{ h>0: s \in hQ_{g}  \right\}$ be defined as in (\ref{eq:beta}). Since we assume that $c \geq \max \lbr f(E),\Vert s \Vert_{\infty} \rbr$, we have that $\beta = \max \left\{ h>0: s \in hQ_{g,c}  \right\}$.  	
\end{remark}
\begin{lemma}
Let $\beta = \max \left\{ h>0: s \in hQ_{g,c}  \right\}$.
	Then, $F_h \neq \emptyset$ if and only if $h \leq \beta $.
\end{lemma}
\begin{theorem}\label{thm:spar-main}
	 Given $s \in \R^E_{>0}$ and $m \in \R^E_{\geq 0}$. For every $0<h \leq \beta $, let $z^*_h$ be an optimal solution of the problem $\min\limits_{z\in F_h} m \cdot z$. Then,
	\begin{equation*}
		m \cdot z^*_{h} \geq m \cdot z^*_{\beta }  \qquad \forall h \leq \beta .
	\end{equation*}
In other words, the polymatroid sparsification problem is equivalent to $\min\limits_{z \in F_{\beta }} m \cdot z.$
\end{theorem}
 \begin{proof}
This can be proved in the same manner as in the proof of Theorem \ref{thm:rein-main}.
 \end{proof}

\subsection{Relationship between $P_f$ and $Q_g$}\label{sec:relationship-poly}
Given a polymatroid function $f$ on $E$, recall that $f$ is normalized, nondecreasing, and submodular. Let the set function $g$ be defined as in equation \eqref{eq:superg}. In this section, we explore several properties of $P_f$ and $Q_g$.

\begin{proposition}\label{thm:p=co-}
	Given a polymatroid function $f$ on $E$. Let $P_f$ be the associated polymatroid of $f$ and $B_f$ be the base polytope of $f$, then
	\begin{equation}
		P_f = (B_f - \R^E_{\geq 0}) \cap  \R^E_{\geq 0}.
	\end{equation}
\end{proposition}

\begin{proof}
	Let $x \in  (B_f- \R^E_{\geq 0}) \cap  \R^E_{\geq 0}$, then \[ 0 \leq x = y - r, \] for some $y\in B_f$ and $r \in \R^E_{\geq 0}$. Hence,
	
	\[ x \leq y. \]
	Because $y \in P_f$, we have  $x \in P_f$. Therefore, \[(B_f - \R^E_{\geq 0}) \cap  \R^E_{\geq 0} \subset P_f. \]
	
	Let $x \in P_f$. Then, there exists $x'\geq 0 $ such that $x+x'$ is maximal in $P_f$. By Lemma \ref{maximal}, we have that \[x+x' \in B_f.\]
	Hence, \[0 \leq x =(x+x')-x' \in B_f - \R^E_{\geq 0}.\] Therefore, $P_f \subset (B_f - \R^E_{\geq 0}) \cap  \R^E_{\geq 0}.$ The proof is completed. 
\end{proof}
\begin{proposition}\label{thm:contrapoly-prop}
Given a polymatroid function $f$ on $E$. Let the set function $g$ be defined as in (\ref{eq:superg}).	Let $c$ be a real number such that $c \geq f(E)=g(E)$. Let $Q_{g}, C_{g}, Q_{g,c}$, and $C_{g,c}$ are defined as in (\ref{eq:contrapoly}), (\ref{eq:base-contrapoly}), (\ref{eq:contrapoly-c}), and (\ref{eq:base-contrapoly-c}), respectively. Then,
\bi
\item[(i)]
\begin{equation}\label{eq:poly-identity1}
	C_{g,c} = C_{g}.
\end{equation}
\item[(ii)]
\begin{equation}\label{eq:poly2-identity1}
	Q_{g,c} = (C_{g,c} + \R^E_{\geq 0}) \cap \lbr x \in \R^E_{\geq 0}: \mathbf{c} \geq x \rbr,
\end{equation}
\item[(iii)]
\begin{equation}\label{eq3:poly-identity1}
	Q_g = C_{g} + \R^E_{\geq 0},
\end{equation}
\ei

\end{proposition}

\begin{proof}
For (i), denote $J := \lbr x \in\R^E_{\geq 0}: x(E)=g(E)\rbr$ and $K_c := \lbr x \in\R^E_{\geq 0}: x \leq \textbf{c} \rbr$. Then, we obtain that $ J \subset K_c$. Hence
\begin{align}
	C_{g,c} = Q_{g,c} \cap J =  Q_{g} \cap K_c \cap J = Q_{g} \cap J = C_g.
\end{align}
By the proof of Theorem \ref{thm:translation} and Proposition \ref{thm:p=co-}, we obtain (ii). Finally, for (iii),

\begin{align*}
	Q_{g}   &= \bigcup\limits_{c > f(E)} Q_{g,c} \\
	&= \bigcup\limits_{c > f(E)} \left((C_{g,c} + \R^E_{\geq 0})  \cap K_c \right)\\
	& =\bigcup\limits_{c > f(E)} (C_{g} + \R^E_{\geq 0})  \cap K_c \\
	&=(C_{g} + \R^E_{\geq 0}) \cap \left( \bigcup\limits_{c > f(E)} K_c \right)\\
	&=C_{g} + \R^E_{\geq 0}.
\end{align*}

\end{proof}

\begin{theorem}\label{thm:hab}
	 Given a polymatroid function $f$ on $E$. Let the set function $g$ be defined as in (\ref{eq:superg}). 	Given $s \in \R^E_{>0}$, let $\alpha$ and $\beta$ be defined as in (\ref{eq:alpha}) and (\ref{eq:beta}). Let $P_f,B_f, Q_g$, and $C_g$ be defined as in (\ref{eq:polymatroid}), (\ref{eq:basepolytope}), (\ref{eq:contrapoly}),  and (\ref{eq:base-contrapoly}). Then, for $h>0$,
	
	\bi

	\item[ (i)]  $s \in hP_f$ if and only if $h \geq \alpha$.
	
	\item[ (ii)] $s \in hQ_g$ if and only if $h \leq \beta$.
	
	\item [ (iii)] $\alpha  \geq \beta$, and $s \in hB_f$ if and only if $h = \alpha=\beta$.
	\ei
\end{theorem}

\begin{proof} 
By definition of $\alpha$ and $\beta$, we have (i) and (ii). For (iii), we have that $s \in \alpha P_f$ and $s \in \beta B_g$, then $\alpha f(E) \geq s(E) \geq \beta g(E)$. Since $f(E)=g(E)$, we have that $\alpha  \geq \beta$. It is well-known that $B_f=C_g$, this is because given that $x(E)=f(E)$, for any subset $A \subseteq E$,
 
\begin{align}
	x(A) \leq f(A) &\Leftrightarrow x(E)-x(E-A) \leq f(A)\\
	&\Leftrightarrow f(E)-x(E-A) \leq f(A)\\
	&\Leftrightarrow x(E-A) \geq f(E)- f(A).
\end{align}

Assume that $h = \alpha=\beta$. Then, $\alpha f(E) = s(E) = g(E)$ and $s \in hB_f = hC_g$. 

Assume that  $s \in hB_f = hC_g$. By definition of $\alpha$ and $\beta$, we have $ \beta \geq h \geq \alpha$.
Note that $hf(E) = s(E)  = hg(E)$, we obtain $ \alpha \geq h \geq \beta$. Therefore, $h = \alpha=\beta$.

\end{proof}

\section{Matroid reinforcement and sparsification}\label{sec:sec-mat-rein-spar}

\subsection{Some properties of strength and fractional arboricity}\label{sec:sd-prob-mat}
In this section, we give some properties for strength and fractional arboricity of matroids. Given a matroid $ M(E,\cI) $ with weights $\si \in \R^E_{>0}$. Let $f$ be the rank function of $M$. Let the set function $g$ be defined as in (\ref{eq:superg}). Let $P_f,B_f, Q_g$, and $C_g$ be defined as in (\ref{eq:polymatroid}), (\ref{eq:basepolytope}), (\ref{eq:contrapoly}),  and (\ref{eq:base-contrapoly}).

\begin{lemma}
	Given a matroid $ M(E,\cI) $ with weights $\si \in \R^E_{\geq 0}$. Let $\cB $ be the base family of $M$.   Let $S_{\si}(M)$ be the strength of $M$. Let $D_{\si}(G)$ be the fractional arboricity of $M$. If $\si(e)=0 $ for some $e\in E$, we have
	\begin{equation}\label{eq:S-edge0}
	S_{\si}(M) = S_{\si}(M\setminus\lbr e \rbr),
	\end{equation}
	and
	\begin{equation}\label{eq:D-edge0}
		D_{\si}(M) = D_{\si}(M\setminus\lbr e \rbr).
	\end{equation}

\end{lemma}
\begin{proof}
For $X \subset E$, we have \[\si(X) = \si(X - \lbr e \rbr),\]  and \[r(X) \geq r(X - \lbr e \rbr) = r_{(M\setminus\lbr e \rbr)}(X - \lbr e \rbr).\]
Then, \[ \frac{\si(X)}{r(X)} \leq \frac{\si(X - \lbr e \rbr)}{r(X - \lbr e \rbr)} = \frac{\si(X - \lbr e \rbr)}{r_{(M\setminus\lbr e \rbr)}(X - \lbr e \rbr)}\]
Based on the definitions of $D_{\si}(M)$, we have (\ref{eq:D-edge0}). Similarly, we obtain (\ref{eq:S-edge0}).
\end{proof}
\begin{proposition}\label{increasing}
	Given a matroid $ M(E,\cI) $ with weights $\si \in \R^E_{\geq 0}$. Let $\cB $ be the base family of $M$.   Let $S_{\si}(M)$ be the strength of $M$. Let $D_{\si}(G)$ be the fractional arboricity of $M$. Then, the functions  $\si \mapsto S_{\si}(M)$ and $\si \mapsto D_{\si}(M)$ are Lipschitz continuous and monotonically increasing.
\end{proposition}
\begin{proof}
	Let $ \si_1, \si_2\in \R^E_{>0}$, and let $X_1$ and $X_2$ be  optimal for $S_{\si_1}(M)$ and $S_{\si_2}(M)$ respectively. Then, without loss of generality, we assume that $S_{\si_1}(M) \leq S_{\si_2}(M)$. So,
	
	\begin{align*}
		S_{\si_2}(M) - S_{\si_1}(M) &\leq  \frac{\si_2(X_1)}{r(E) - r(E - X_1)} - \frac{\si_1(X_1)}{r(E) - r(E - X_1)} \\
		& = \frac{1}{r(E) - r(E - X_1)}\sum\limits_{e \in X_1} (\si_2(e)-\si_1(e)) \\
		&   \leq  \frac{1}{r(E) - r(E - X_1)} \sum\limits_{e \in X_1} |\si_2(e)-\si_1(e)|\\
		&\leq \sum\limits_{e \in X_1} |\si_2(e)-\si_1(e)| \leq \sum\limits_{e \in E} |\si_2(e)-\si_1(e)| \\
		&\leq |E| \Vert \si_2-\si_1 \Vert_{\infty}.
	\end{align*}
	where $\Vert\cdot\Vert_{\infty}$ is the maximum norm on $\R^E$. So,  the function $\si \mapsto S_{\si}(G)$ is Lipschitz continuous. We have a similar argument for $D_{\si}(G).$
	
	For the monotonicity, assume that $\si \leq \si'$, then $\si(A) \leq \si'(A)$ for all $A \subset E.$ Therefore $ S_{\si}(G) \leq S_{\si'}(G)$ and  $ D_{\si}(G) \leq D_{\si'}(G)$.
\end{proof}

\subsection{Matroid reinforcement}\label{sec:rein-mat}

Given a matroid $ M(E,\cI) $ with weights $\si \in \R^E_{>0}$. Let $f$ be the rank function of $M$.  Let $\alpha$ be defined as in (\ref{eq:alpha}).
 By Theorem \ref{thm:rein-main}, the problem (\ref{eq:rein-poly}) is equivalent to (\ref{eq:rein-poly-reduced}).
By Remark \ref{ehandps}, this is equivalent to 
\begin{equation}\label{pbasic2}
	\min\left\{m \cdot z : z \text{ is maxmimal in }  (\alpha P_f)^{\si} \right\}.
\end{equation}
Since $(\alpha P_f)_{\si}$ is a polymatroid, we can use the greedy algorithm to solve this problem. To implement the greedy algorithm, we have that the oracle
\begin{equation}\label{oracle3}
	\max \left\{ \epsilon : z+\epsilon\ones_{\left\{ j\right\}} \in (\alpha P_f)^{\si} \right\},
\end{equation}
is equivalent to 
\begin{equation}\label{oracle4}
	\max \left\{ \epsilon : \si+z+\epsilon\ones_{\left\{ j\right\}} \in \alpha P_f \right\}.
\end{equation}
Note that the oracle (\ref{oracle4}) is equivalent to 

\begin{equation}\label{oracle2}
	\min \left\{ \alpha f(A) - (\si +z)(A): j \in A \subset E \right\}.
\end{equation}
Assume that we have a method to implement the oracle (\ref{oracle2}). Suppose $m_{j_1} \leq m_{j_2} \leq \dots \leq m_{j_k}$ where $k = |E|$. We introduce  Algorithm \ref{al:hmt-in} for the matroid reinforcement problem (\ref{eq:hmt1}).
\begin{algorithm}
	\caption{Algorithm for the matroid reinforcement problem}\label{al:hmt-in}
	\hspace*{\algorithmicindent} \textbf{Input:}  $M= (E,\cI),\si,\alpha$
	
	\hspace*{\algorithmicindent} \textbf{Output:} Optimal  $z$
	\begin{algorithmic}[1] 
		\STATE $z \leftarrow 0$
		\FOR {$i \in \left\{ 1,2,\dots,k \right\}$}
		\STATE $z(j_i) \leftarrow  z(j_i) + \min \left\{\alpha f(A) - (\si+z)(A): j_i \in A \subset E \right\}$
		\ENDFOR
		\RETURN z
	\end{algorithmic}
\end{algorithm}

\begin{theorem}
	 Let $m \in \R^E_{\geq 0}$ be the cost of increasing elements per unit. Let $D_{\si}(M)$ be the fractional arboricity and $\alpha = D_{\si}(M)$. Then, after each iteration of step (3) in Algorithm \ref{al:hmt-in}, the matroid $M$ with the new weights $\si+z$ has the fractional arboricity remaining unchanged. When the algorithm terminates and outputs the optimal $z$, the matroid $M=(E,\cI)$ with weights $(\si+z)$ is homogeneous.
\end{theorem}

\begin{proof}
	After each iteration of step (3) of Algorithm \ref{al:hmt-in}, the weights are increasing. Then, by Proposition \ref{increasing}, the fractional arboricity is increasing as well. In contrast, since the new weights are always in $\alpha P_f$, the fractional arboricity never exceeds $\alpha $ during the process by Theorem \ref{thm:hab}. Therefore, the fractional arboricity remains unchanged during the entire algorithm. When the algorithm terminates, $\si+z$ is maximal in $\alpha P_f$, this means that $\si+z \in \alpha B_f =\alpha \co(\cB)$. Therefore, the matroid $M=(E,\cI)$ with weights $(\si+z)$, with the optimal $z$, is homogeneous.
\end{proof}

\subsection{Matroid sparsification}\label{sec:spar-mat}

Given a matroid $ M(E,\cI) $ with weights $\si \in \R^E_{>0}$ and the rank function $f$.  Let $\beta$ be defined as in (\ref{eq:beta}).
Follow the same method as in Section \ref{sec:rein-mat}, we obtain Algorithm \ref{al:hmt-de} for the matroid sparsification problem (\ref{eq:hmt2}).

\begin{algorithm}
	\caption{Algorithm for the matroid sparsification problem}\label{al:hmt-de}
	\hspace*{\algorithmicindent} \textbf{Input:}  $M= (E,\cI),\si,\beta$
	
	\hspace*{\algorithmicindent} \textbf{Output:} Optimal  $z$
	\begin{algorithmic}[1] 
		\STATE $z \leftarrow 0$
		\FOR {$i \in \left\{ 1,2,\dots,k \right\}$}
		\STATE $z(j_i) \leftarrow  z(j_i) + \min \left\{ (\si -z)(A) - \beta g(A) : j \in A \subset E \right\}$
		\ENDFOR
		\RETURN z
	\end{algorithmic}
\end{algorithm}

\begin{theorem}
	Let $m \in \R^E_{\geq 0}$ be the cost of decreasing elements per unit. Let $S_{\si}(M)$ be the strength and $\beta = S_{\si}(M)$. Then, after each iteration of step (3) in Algorithm \ref{al:hmt-de}, the matroid with the new weights $\si-z$ has the strength remaining unchanged. When the algorithm terminates and outputs the optimal $z$, the matroid $M=(E,\cI)$ with weights $(\si-z)$ is homogeneous.
\end{theorem}

\begin{proof}
	After each iteration of step (3) of Algorithm \ref{al:hmt-de}, the weights are decreasing. Then, by Proposition \ref{increasing}, the strength is decreasing as well. In contrast, since the new weights are always in $\beta Q_g$, the strength never go below $\beta$ during the process by Theorem \ref{thm:hab}. Therefore, the strength remains unchanged during the entire algorithm. When the algorithm terminates, $\si-z$ is maximal in $\beta Q_g$, this means that $\si-z \in \beta C_g =\beta \co(\cB)$. Therefore, the matroid $M=(E,\cI)$ with weights $(\si-z)$, with the optimal $z$, is homogeneous.
\end{proof}

\section{Application for graphs}\label{sec:app-graph}

In this section, we consider the matroid reinforcement and matroid sparsification problems for graphic matroids.

\subsection{An algorithm for computing  fractional arboricity}\label{sec:app-d}
We consider an undirected, connected, and weighted graph $G= (V,E,\si)$ with edge weights $\si \in \R^E_{>0}$. Let $f$ be the rank function of the graphic matroid associated with $G$.
 Let $\alpha$ be defined as in (\ref{eq:alpha}), we recall that $\alpha = D_{\si}(G)$ where the fractional arboricity of $G$ is defined as:
 \begin{equation}
 	D_{\si}(G) := \max \left\{ \frac{\si(X)}{f(X)} : X \subseteq E, r(X)>0 \right\}.
 \end{equation}
 
\begin{theorem}\label{thm:def-D}
	 Let $G =(V,E,\si)$ be a weighted connected graph. Let $D_{\si}(G)$ be the fractional arboricity of $G$. Let $\cH$ be the set of all vertex-induced connected subgraphs of $G$ that contain at least one edge. Let $E_B$ denote the set of edges that have both endpoints in each set $B \subseteq V$. Then
	\begin{equation}\label{eq:def-D}
		D_{\si}(G) = \max \left\{ \frac{\si(E_B)}{|B|-1}:B \subseteq V, |B| \geq 2 \right\} = \max\limits_{H \in \cH} \frac{\si(E_{H	})}{|V_H|-1}.
	\end{equation}

\end{theorem}

\begin{proof}
We have three definitions for fractional arboricity: 
\[
D_1 := D_{\si}(G),\qquad D_2 :=  \max \left\{ \frac{\si(E_B)}{|B|-1}:B \subseteq V, |B| \geq 2 \right\},\qquad D_3  := \max\limits_{H \in \cH} \frac{\si(E_{H	})}{|V_H|-1}.
\]

	Let $H' =(V_{H'},E_{H'})$ be a subgraph which optimizes $D_3$. Choose $X = E_{H'}$, then $f(X) =|V_{H'}|-1$. Thus,  we obtain $D_3 \leq D_1$. Choose $B = V_{H'}$, then $D_3 \leq D_2$.
	
Let $X$ be an optimizer for $D_1$. Then $H = (V_X,X)$ be the edge-induced subgraph generated by $X$.
Suppose $H$ has $m$ connected components $H_1,H_2,\dots, H_m$. We want to show that \[\frac{\si(X)}{f(X)} \leq \max\limits_{i = 1,\dots,m} \frac{\si(E_{H_i})}{|V_{H_i}| -1}.\] 
Notice  that for $a_1, \dots,a_n >0$ and $b_1,\dots,b_n>0$, we have
\[
\sum_{i=1}^n a_i \le \left(\max\limits_{i =1,\dots,n} \frac{a_i}{b_i}\right)\sum_{i=1}^n b_i. 
\]
Therefore,
\[\frac{a_1+a_2+\dots + a_n}{b_1+b_2+\dots + b_n} \leq \max\limits_{i =1,\dots,n} \frac{a_i}{b_i}.\]
In particular, 
\begin{align}\label{eq:dd3}
	\frac{\si(X)}{f(X)}  &\leq \frac{\si(E_{H_1})+\dots+\si(E_{H_m})}{(|V_{H_1}|-1)+\dots+(|V_{H_m}|-1)} & \notag \\     
	&\leq \max\limits_{i =1,\dots,m} \frac{\si(E_{H_i})}{|V_{H_i}|-1}. & \notag 
\end{align}
Thus, $D_1 \leq D_3$.

Let $B$ be an optimizer for $D_2$. Then $K = (B,E_B)$ be the vertex-induced subgraph generated by $B$.
Suppose $K$ has $m$ connected components $K_1,K_2,\dots, K_m$ where each component has at least one edge. We want to show that \[\frac{\si(E_B)}{|B| -1} \leq \max\limits_{i = 1,\dots,m} \frac{\si(E_{K_i})}{|V_{K_i}| -1}.\] 
	We have, 
	\begin{align}\label{eq:dd3}
		\frac{\si(E_B)}{|B| -1}  &\leq \frac{\si(E_{K_1})+\dots+\si(E_{K_m})}{|V_{K_1}|+\dots+|V_{K_m}|-1} & \notag \\     
		&\leq \frac{\si(E_{K_1})+\dots+\si(E_{K_m})}{|V_{H_1}|+\dots+|V_{K_m}|-m}  & \notag \\ 
		&\leq \max\limits_{i =1,\dots,m} \frac{\si(E_{K_i})}{|V_{K_i}|-1}. & \notag 
	\end{align}
	Thus, $D_2 \leq D_3$.
\end{proof}
 It is well-known that the polytope $P_f$ defined as in (\ref{eq:polymatroid}) can be described as follows:
	\begin{equation}\label{poly4}
		\begin{array}{ll}
			x(E_B) \leq |B|-1\qquad \forall B \text{ such that } \emptyset \neq B \subseteq V; \\
			x \geq 0.
		\end{array}
	\end{equation}
	where $E_B$ denotes the set of edges that have both endpoints in $B$. This fact gives rise to  a second proof for the first equality in (\ref{eq:def-D}).

Next, we give an algorithm for computing the fractional arboricity $D_{\si}(G)$. There is a well-known method for dealing with quotients like $D_{\si}(G)$. It is described as follows.
Recall that $\alpha = D_{\si}(G)$ and let $b>0$. By (\ref{eq:def-D}), we have

\begin{align}
	\alpha  <  b &   \Leftrightarrow  \max \left\{ \frac{\si(E_B)}{|B|-1}:B \subseteq V, |B| \geq 2 \right\}  < b \notag \\ 
	&   \Leftrightarrow   \frac{\si(E_B)}{|B|-1} < b \qquad \forall B \text{ such that } B \subseteq V, |B|\geq 2 \notag\\
	& \Leftrightarrow    b(|B|-1) - \si(E_B) > 0 \qquad \forall B \text{ such that } B \subseteq V, |B|\geq 2 \notag\\
	& \Leftrightarrow   g(b):=  \min \left\{b(|B|-1) - \si(E_B): B \subseteq V,|B|\geq 2\right\} > 0.
\end{align}
Note that we also have similar statements:
\begin{equation}\label{eq:kb-eq-gb-2}
	\alpha = b \Leftrightarrow  g(b) = 0,
\end{equation}
\begin{equation}\label{eq:kb-eq-gb-3}
	\alpha > b \Leftrightarrow  g(b) < 0.
\end{equation} 

Next, we assume that we have a lower bound $b$ of $\alpha$, in other words, $\alpha \geq b$. Then, we have $g(b) \leq 0$, let's consider two cases:
\begin{itemize}
	\item If $g(b) = 0$, then $\alpha =  b$.
	\item If  $g(b) < 0$. Let  $B^{*}$ be a minimizer of $g(b)$, then  
	\begin{align*}
		&b(|B^{*}|-1) - \si(E_{B^{*}}) <0 \\
		&\Leftrightarrow  b':= \frac{\si(E_{B^{*}})}{|B^{*}|-1} >b.
	\end{align*}
\end{itemize}
Then, $b'$ is a better lower bound of $\alpha$. Based on this argument, we obtain an algorithm for computing $\alpha$:

\begin{algorithm}
	\caption{Algorithm for computing $\alpha$}\label{al0}
	\hspace*{\algorithmicindent} \textbf{Input:}  $G= (V,E,\si)$
	
	\hspace*{\algorithmicindent} \textbf{Output:} $\alpha$
	\begin{algorithmic}[1] 
		\STATE $B \leftarrow V$
		\STATE $b \leftarrow \si(E)/(|V|-1)$
		\WHILE {$ g(b) < 0$}
		\STATE $B \leftarrow$ a minimizer of $g(b)$
		\STATE $b \leftarrow \si(E_B)/(|B|-1)$
		\ENDWHILE
		\RETURN b, B
	\end{algorithmic}
\end{algorithm}

The following lemma shows that Algorithm  \ref{al0} takes at most $n=|V|$ steps.

\begin{lemma}\label{convergence}
	Assume that $g(b_0)<0$ and  $B_0$ is a minimizer of $g(b_0)$. Denote \[b_1:= \frac{\si(E_{B_0})}{|B_0|-1}.\] Let $B_1$ be a minimizer of $g(b_1)$. If $g(b_1)<0$, then  $|B_1| < |B_0|$.
\end{lemma}

\begin{proof}
	Because $g(b_0)<0$ and $g(b_1)<0$, we have
	\begin{equation*}
		\frac{\si(E_{B_1})}{|B_1|-1} >  b_1= \frac{\si(E_{B_0})}{|B_0|-1} > b_0.
	\end{equation*}
	Then,
	\begin{align*}
		0 &< \si(E_{B_1}) - b_1(|B_1|-1) &\\
		& =  \si(E_{B_1}) - b_0(|B_1|-1)+ b_0(|B_1|-1)  - b_1(|B_1|-1) &\\
		& \leq  \si(E_{B_0}) - b_0(|B_0|-1)+ b_0(|B_1|-1)  - b_1(|B_1|-1) & (\text{because } B_0 \text{ is a minimizer of } g(b_0))\\
		& = b_1(|B_0|-1) - b_0(|B_0|-1)+ b_0(|B_1|-1)  - b_1(|B_1|-1)&\\
		& = (b_0-b_1)(|B_1|-|B_0|).
	\end{align*}
	Since  $ b_0-b_1 <0$, we obtain $|B_1| < |B_0|$.
\end{proof}
By Lemma \ref{convergence}, we conclude that Algorithm \ref{al0} will compute $D_{\si}(G)$ within at most $|V|$ iterations. The remaining work is to find a way to compute  $g(b)$ given that $g(b) \leq 0$. Denote $x := \si/b$, then finding $g(b)$ is equivalent to finding
\begin{equation}\label{eq:g(b)}
	\min \left\{(|B|-1) - x(E_B): B \subseteq V,|B|\geq 2\right\}.
\end{equation}

Let's recall the problem of finding a most-violated inequality for the description (\ref{poly4}) of $P_f$:
\begin{equation}\label{eq:most-violated}
	\min \left\{(|B|-1) - x(E_B): \emptyset \neq B \subseteq V \right\}.
\end{equation}
In general, the problem (\ref{eq:g(b)}) and the problem (\ref{eq:most-violated}) are not equivalent. But note that we have $g(b) \leq 0$, and when $|B| = 1$ we have $|B|-1 - x(E_B) = 0$. Therefore, in our case, these two problems are equivalent. A network flow algorithm for this problem is  known \cite{cunninghamminimal}. 

\subsubsection*{A Cunningham's minimum cut formulation.}
In \cite{cunninghamminimal}, Cunnningham gives an algorithm for finding a most-violated inequality for the description (\ref{poly4}). We recall the algorithm for completeness.

Given $x \in \R^E_{>0}$. As described in \cite{cunninghamminimal}, we construct a capacitated undirected graph $G'$ from $G = (V,E)$ as follows.
\bi
\item The vertex set of $G'$ is $V \cup \left\{ r,s\right\}$.
\item Every edge $e$ of $G$ is an edge of $G'$ with the same endpoints and it has capacity $\frac{1}{2}x(e)$.
\item  For every $v \in V$, there is an edge connecting $v$ and $s$ and it has capacity $1$.
\item For every $v \in V$, let $\delta(v) := \left\{ e\in E: e \text{ is incident with } v  \right\}$, there is an edge connecting $v$ and $r$ and it has capacity $\frac{1}{2}x(\delta(v))$.
\ei
Denote the capacity of edges in $G'$ by $w$ and $B = V \setminus A$. Consider a cut separating $r$ and $s$ in $G'$ determined by 

\[ A \cup \left\{ r\right\} \text{ for any } A \subset V.\]
Then, the value of this cut is:
\begin{align*}
	& \sum\limits_{e =\left\{ a,b \right\}:a \in A, b \in B} w(e) + \sum\limits_{e =\left\{ r,b \right\}: b \in B} w(e)+\sum\limits_{e =\left\{ a,s \right\}:a \in A} w(e) \\ 
	& = \sum\limits_{e =\left\{ a,b \right\}:a \in A, b \in B} \frac{1}{2}x(e) +  \sum\limits_{b \in B} \frac{1}{2}x(\delta(b)) + |A|\\
	& = x(E)-x(E_A)+|A|.
\end{align*}
where  $E_A$ is the set of edges that have both endpoints in $A$.
Therefore, the problem (\ref{eq:most-violated}) is equivalent to find a minimum cut  $A \cup \left\{ r\right\}$  satisfying $A \neq \emptyset$. To overcome the condition $A \neq \emptyset$, we just need to solve $|V|$ minimum cut
problems, each one determined by changing the capacity of an edge connecting $v$ and $r$ to $\infty$. This is because any cut  $A \cup \left\{ r\right\}$ with $A  \neq \emptyset$ will not use at least one of these edges.

Run-time complexity for Algorithm \ref{al0}: $|V|^2$ minimum-cut calculations.  This is because the problem of finding a most-violated inequality for the description (\ref{poly4}) of $P_f$ takes  $|V|$ minimum-cut calculations and the number of iterations is at most $|V|$.

\subsection{Matroid reinforcement for graphs}\label{sec:app-rein}
In this section, we give a detailed method to implement Algorithm \ref{al:hmt-in}.
\subsubsection*{Cunningham's minimum cut formulation modification}
In \cite{cunninghamoptimal}, Cunningham gives a network-flow algorithm that solves (\ref{oracle2}). Given $x \in P_f$ and $j \in E$, we want to solve
\begin{equation}\label{eq:ora-10}
	\min \left\{ f(A) - x(A): j \in A \subset E \right\}.
\end{equation}
Because we have that  $x \in P_f$, then a set $A$ that minimizes $f(A) - x(A)$ over $j \in A \subseteq E $ can be required to be of the form $E_B$. Hence, Cunningham modified his minimum cut formulation  as follows.

We construct an undirected capacitated graph $G'$ from $G=(V,E)$ as follows.
\bi 
\item The vertex set for $G'$ is $V \cap \left\{ r,s\right\}$, where $r$ and $s$ are new vertices that will be the source and sink respectively.
\item Every edge $e$ of $G$ is an edge of $G'$ with the same endpoints and it has capacity $\frac{1}{2}x(e)$.
\item  For every $v \in V$, there is an edge connecting $v$ and $s$ and it has capacity $1$.
\item For every $v \in V$, let $\delta(v) := \left\{ e\in E: e \text{ is incident with } v  \right\}$, there is an edge connecting $v$ and $r$. It has capacity $\frac{1}{2}x(\delta(v))$ if $v$ is not an endpoint of $j$ and has capacity $\infty$ if $v$ is an endpoint of $j$.
\ei
As shown in \cite{cunninghamoptimal}, (\ref{eq:ora-10}) can be recovered from the value of  a minimum cut seperating $r$ and $s$ in $G'$ and the edges of a minimum cut (after removing any edges incident with $r$ or $s$) form a minimizer for (\ref{eq:ora-10}).

Run-time complexity for Algorithm \ref{al:hmt-in}: $|V|^2 + |E|$ minimum-cut calculations. This is due to finding $D_{\si}(G)$ and the number of iterations is $|E|$ where each iteration takes 1 minimum-cut calculation.

\subsection{Matroid sparsification for graphs}\label{sec:app-spar}
In this section, we provide a detailed method to implement Algorithm \ref{al:hmt-de}.
Given $x \in Q_g$ and $j \in E$, we want to solve
\begin{equation*}
	\min \left\{ x(A) - g(A): j \in A \subset E \right\}.
\end{equation*}
This is equivalent to
\begin{equation}\label{eq:spargraph0}
	\min \left\{ x(A) + f(E-A): j \in A \subset E \right\}.
\end{equation}
By a change of variable $B:= E-A$, the problem (\ref{eq:spargraph0}) is equivalent to \begin{equation}\label{eq:spargraph}
	\min \left\{  f(B) -x(B): j \notin B \subset E \right\}.
\end{equation}
We can deal with this problem by deleting the edge $j$ from $G$. Let $G' = (V_{G'},E_{G'}) := G - \lbr j\rbr$, then (\ref{eq:spargraph}) is equivalent to
\begin{equation}
	\min \left\{  f_{G'}(B) -x(B): B \subset E_{G'} \right\}.
\end{equation}
In fact, this is known as the attack problem, see \cite{cunninghamoptimal}. Cunningham gives an algorithm for it in \cite{cunninghamoptimal}. The algorithm requires $|E|$ minimum-cut calculations.

Run-time complexity for Algorithm \ref{al:hmt-de}: $|V||E| + |E|^2$ minimum-cut calculations. This is due to finding $S_{\si}(G)$ (we use the algorithm in \cite{cunninghamoptimal}, it requires $|V||E|$ minimum-cut calculations) and the number of iterations is $|E|$ where each iteration takes $|E|$ minimum-cut calculations.

\subsection{An algorithm for computing spanning tree modulus}\label{sec:app-mod}

Let $G= (V,E,\si)$ be a weighted graph with edge weights $\si \in \R^E_{>0}$. Let $\cB$ be the family of spanning trees of $G$. In \cite{polynomial}, they propose an algorithm for computing spanning tree modulus based on $E_{ max}$ (defined as in (\ref{eq:emax})) using Cunningham's algorithm for computing the strength of a graph.

Applying results in \cite{huyfulkerson}, we suggest an algorithm for computing spanning tree modulus using Algorithm \ref{al0} based on  $E_{ min}$ (defined as in (\ref{eq:emin})).

Any subgraph which is optimal for the fraction arboricity problem (\ref{eq:def-D}) is said to be a D-optimal subgraph.
First, we find a D-optimal subgraph $H$ of $G$ using Algorithm \ref{al0}. Then,  Theorem \ref{emin} show that $\si^{-1}\eta^*$ takes the value $(|V_H|-1)/\si(E_{H	})$ on all edges in $H$. Now, shrink $H$ to a vertex, this results in a shrunk graph $G_1 = G/H$. Next, find a D-optimal subgraph $H_1$ of $G_1$ and $\si^{-1}\eta^*$ (for the spanning tree modulus of $G$) takes the value $(|V_{H_1}|-1)/\si(E_{H_1})$ on the edges of $H_1$. Repeat this procedure, each time computes $\si^{-1}\eta^*$ for at least one edge. Thus, after finite iterations, $ \si^{-1}\eta^*$ will be computed for all edges. A more detailed description of this algorithm is described in Algorithm \ref{al5}.
\begin{algorithm}
	\caption{Using D-optimal subgraphs to compute spanning tree modulus}\label{al5}
	\hspace*{\algorithmicindent} \textbf{Input:}  $G= (V,E,\si)$
	
	\hspace*{\algorithmicindent} \textbf{Output:} $\eta^*$
	\begin{algorithmic}[1] 
		\WHILE{$G$ is nontrivial}
		\STATE compute $D_{\si}(G)$ and a D-optimal subgraph $H=(V_H,E_H)$
		\FORALL{ $e \in E_H$}
		\STATE $\eta^*(e) \leftarrow \si(e)(D_{\si}(G))^{-1}$
		\ENDFOR
		\STATE $G \leftarrow G/H$
		\ENDWHILE
		\RETURN $\eta^*$
	\end{algorithmic}
\end{algorithm}

Run-time complexity for Algorithm  \ref{al5}: $|V|^3$ minimum-cut calculations. This is because solving $D_{\si}(G)$ takes $|V|^2$ minimum-cut calculations and the number of iterations is at most $|V|$.

\newpage

\bibliographystyle{acm}
\bibliography{main}
\def\cprime{$'$}
\nocite{*}
\end{document}